\documentclass[11pt]{article}

\usepackage{amsthm,amsfonts,amssymb,amsmath,color}
\usepackage{braket}
\usepackage{bm}

\numberwithin{equation}{section}

\renewcommand\d{\partial}
\renewcommand\a{\alpha}

\renewcommand\o{\omega}
\newcommand\s{\sigma}
\renewcommand\t{\tau}

\def\t{\tau}
\def\O{\Omega}

\def\ess{{\rm ess}}

\def\epsilon{\varepsilon}
\def\e{\varepsilon}

\newcommand\br{\begin{rem}}
\newcommand\er{\end{rem}}
\newcommand\bp{\begin{pmatrix}}
\newcommand\ep{\end{pmatrix}}
\newcommand\be{\begin{equation}}
\newcommand\ee{\end{equation}}
\newcommand\ba{\begin{equation}\begin{aligned}}
\newcommand\ea{\end{aligned}\end{equation}}

\newcommand\nn{\nonumber}

\newcommand{\uu}{{\mathbf u}}
\newcommand{\vv}{{\mathbf v}}
\newcommand{\ff}{{\mathbf f}}

\newcommand{\ww}{{\mathbf w}}

\newcommand{\vh}{{\mathbf h}}

\newcommand{\vr}{\varrho}

\newcommand{\vu}{\vc{u}}
\newcommand{\vf}{\vc{f}}

\newcommand{\vc}[1]{{\bf #1}}

\newcommand{\dive}{{\rm div\,}}

\newtheorem{defi}{Definition}[section]
\newtheorem{theorem}[defi]{Theorem}
\newtheorem{proposition}[defi]{Proposition}
\newtheorem{lemma}[defi]{Lemma}

\newtheorem{remark}[defi]{Remark}

\numberwithin{equation}{section}
\hoffset - 1 cm

\usepackage{ifthen}

\makeatletter
\renewcommand\@biblabel[1]{#1.}
\makeatother

\newlength{\bibitemsep}
\newlength{\bibparskip}
\let\oldthebibliography\thebibliography
\renewcommand\thebibliography[1]{%
  \oldthebibliography{#1}%
  \setlength{\parskip}{\bibitemsep}%
  \setlength{\itemsep}{\bibparskip}%
}

\begin{document}

\title{Homogenization of Non-homogeneous Incompressible Navier-Stokes System in Critically Perforated Domains}

\author{Jiaojiao Pan\footnote{School of Mathematics, Nanjing University, 22 Hankou Road, Gulou District, Nanjing 210093, China, {\tt panjiaojiao@smail.nju.edu.cn.}}}

\date{}

\maketitle

\noindent
{\bf Abstract.} In this paper, we study the homogenization of 3D non-homogeneous  incompressible Navier-Stokes system in perforated domains with holes of critical size. The diameter of the holes is of size $\e^{3}$, where $\e$ is a small parameter measuring the mutual distance between the holes. We show that when $\e\to 0$, the velocity and density converge to a solution of the non-homogeneous incompressible Navier-Stokes system with a friction term of Brinkman type.



\renewcommand{\refname}{REFERENCES}

\section{Introduction}
\label{Introduction}
Homogenization of different fluid flows in physical domains has been widely considered, where the domains are perforated by a large number of tiny holes. When the number of holes tends to infinity and their size tends to zero, the typical diameter and mutual distance of these holes become the main factors in the asymptotic behavior of fluid flows. The limit system that describes the limit behavior of fluid flows is called homogenized system defined in homogeneous domains without holes.

Tartar \cite{Tartar1} firstly considered the case where the holes’ mutual distance is of the same order as their radius for Stokes equations and derived Darcy's law in the homogenization limit. Allaire \cite{ALL-NS1, ALL-NS2} provided a systematical study of stationary Stokes and Navier-Stokes systems, where the porous medium is modeled as the periodic repetition of an elementary cell of size $\e$, in which the hole is of size $a_\e$. The homogenized system is determined by the ratio $\s_{\e}$ given as
\begin{equation}\label{1-sigma}
\sigma _\varepsilon: = \Big(\frac{\varepsilon^d}{a_\varepsilon^{d-2}}\Big)^{\frac{1}{2}},  \ d \geqslant 3;\quad{\sigma _\varepsilon}: = \varepsilon \left| \log \frac{{a_\varepsilon }}{\varepsilon} \right|^{\frac{1}{2}}, \ d = 2.\nn
\end{equation}
Roughly speaking, if $\lim_{\e\to 0}\s_{\e}=0$ corresponding to the case of supercritical size of holes, the asymptotic limit behavior is governed by Darcy's law. If $\lim_{\e\to 0}\s_{\e}=\infty$ corresponding to the case of subcritical size of holes, the limit system coincides with the original system.  If $\lim_{\e\to 0}\s_{\e}=\s_{*}\in (0,+\infty)$ corresponding to critical size of holes, the asymptotic limit behavior gives rise to Brinkman's law, which is a combination of Darcy's law and the original equations. Later, Lu \cite{L} gave a unified approach for Stokes equations in perforated domains. Such an idea of unified approach was also used in Jing \cite{Wenjia Jing,Wenjia Jing1} for the study of Laplace equations and Lamé systems in perforated domains.

For evolutionary incompressible Navier-Stokes system, Mikeli\'{c} \cite{Mik} studied the homogenization process where the size of holes is proportional to the mutual distance of holes and the result of homogenization is Darcy’s law. Feireisl, Namlyeyeva and Ne\v casov\' a \cite{FeNaNe} studied the critical case and derived Brinkman’s law. Recently, Lu and Yang \cite{LY} considered supercritical and subcritical size of holes. 

For compressible Navier-Stokes system,  Masmoudi \cite{Mas-Hom} studied the homogenization in the case where the size of holes is proportional to the mutual distance of holes and derived Darcy’s law in the limit. For the case of very small holes, there are series of studies \cite{BO2,DFL,FL1,Lu-Schwarz18}, which show the limit system does not change. Still for very small holes, two-dimensional stationary compressible Navier-Stokes equations, Ne\v casov\' a and Pan \cite{NP} showed that the limit system is the same as the original one. Similar results were given by Ne\v casov\' a and Oschmann \cite{NF} for evolutionary case.  In \cite{Richard M. Hofer1,BO1}, assuming that the Mach number decreases fast enough, Darcy’s law and Brinkman’s law are derived in the limit for supercritical case and critical case respectively.

Moreover, Feireisl, Novotn\'y and Takahashi \cite{FNT-Hom} studied homogenization of the full Navier-Stokes-Fourier system, where the diameter of the holes is proportional to their mutual distance and the result of homogenization is Darcy’s law. Lu and Pokorn\'y \cite{Lu-Pokorny} studied stationary compressible Navier-Stokes-Fourier system with subcritical size of holes, where the limit system remains the same as original one. Feireisl, Lu and Sun \cite{FLS} studied a non-homogeneous incompressible and heat conducting fluid confined to a 3D domain perforated by tiny holes of critical size, and showed that the limit system contains a friction term of Brinkman type.

Recently, Lu and Qian \cite{LY-Qian} considered the homogenization of evolutionary incompressible viscous non-Newtonian flows of Carreau-Yasuda type in porous media, where the size of holes is proportional to the mutual distance of holes. The result of homogenization is Darcy’s law. Then Lu and Oschmann \cite{Lu-Oschmann} generalized the results to the cases where the mutual distance between the holes is measured by a small parameter $\e>0$ and the size of the holes is $\e^{\a}$, $\a\in (1,\frac{3}{2})$. Moreover, H\"{o}fer \cite{Richard M. Hofer} studied the homogenization of the Navier-Stokes equations in perforated domains in the inviscid limit with different sizes of the holes. It is more interesting and more difficult to derive the convergence rate from the original system to the limit system. Allaire \cite{ALL-NS1,ALL-NS2} gave convergence rates for Stokes equations for different sizes of the holes. Very recently, Shen \cite{Shen} gave the sharp convergence rate for Stokes equations to Darcy’s Law.

In this paper, we consider the homogenization of 3D non-homogeneous incompressible Navier-Stokes system, where  the diameter of holes is of size $\e^{3}$ and the mutual distance of holes is proportional to $\e$ (critical case). We show that the limit system in momentum equation contains a term of Brinkman type given by viscosity and perforation properties. The limit continuity equation is the same as the original one. 
 
\subsection{Problem formulation}
\label{Problem formulation}
Let $\Omega \subset \mathbb R^{3}$ be a bounded domain of class $C^2$. We consider a family of $\e$-dependent perforated domains
$\{ \Omega_\e \}_{\e > 0}$,
\be\label{1.1}
\Omega_\e = \Omega \setminus \bigcup_{k\in K_\e} T_{\e, k},\ \ K_\e:=\{k\in \mathbb{Z}^3|\ \e \overline{C}_k\subset\O\},
\ee
where the sets $T_{\e,k}$ represent holes and 
 \be
C_k := (-\frac{1}{2},\frac{1}{2})^3+k,\ \ k\in \mathbb{Z}^3.\nn
\ee
Suppose we have the following property concerning the distribution of the holes:
 \be\label{1.2}
B(x_{\e,k}, \frac{1}{2}a_{\e} ) \subset \subset T_{\e,k} := x_{\e,k} + \ a_{\e} \overline{U}_{k,\e}  \subset\subset B(x_{\e,k}, \frac{3}{4}a_{\e} ) \subset \subset\e C_k\subset\subset\O.
\ee
Here $B(x,r)$ denotes the open ball centered at $x$ with radius $r$ in $\mathbb R^{3}$. $x_{\e,k}\in \mathbb R^{3}$ represent the locations of holes, $a_{\e}$ is the size of the holes and $\{U_{k,\e}\}_{\e>0,\; k\in K_\e }$ are uniformly $C^{2+\nu}$ simply connected domains satisfying 
\be
\left\{|x|<\frac{1}{2}\right\}\subset U_{k,\e}\subset \overline{U}_{k,\e}\subset \left\{|x|<\frac{3}{4}\right\}, \quad \textup{for any }\e, k.\nn
\ee   

We consider the case $a_{\e}=\e^{3}$, which means $\lim_{\e\to 0}\s_{\e}=1$ (critical case). The total number of the holes $|K_\e|$ can be estimated as
\be
|K_{\e}| \leq  \frac{|\Omega|}{\e^3}(1+o(1)).\nn
\ee

Consider non-homogeneous incompressible Navier-Stokes system in $\Omega_{\varepsilon}$:
\begin{eqnarray}\label{1.3}
\left\{
\begin{aligned}
&\d_{t}\vr_{\e}+\dive (\vr_{\e} \vu_{\e}) = 0, \quad &\textup{in }&(0,T)\times\Omega_{\varepsilon},\\
&\d_{t}(\vr_{\e} \vu_{\e})+ \dive (\vr_{\e} \vu_{\e} \otimes \vu_{\e})-\mu \Delta \vu_{\e} +\nabla p_{\e} =\vr_{\e} \vf_{\e}, \quad & \textup{in }&(0,T)\times\Omega_{\varepsilon},\\
&\dive \vu_{\e} = 0, \quad &\textup{in }&(0,T)\times\Omega_{\varepsilon},\\
&\vu_{\e} = 0, \quad &\textup{on }&(0,T)\times\d\Omega_{\varepsilon},\\
&\vr_{\e}|_{t=0} = \vr_{\e}^{0},\quad (\vr_{\e}\vu_{\e})|_{t=0} = \vr_{\e}^{0}\vu_{\e}^{0},
\end{aligned}
\right .
\end{eqnarray}
where $\vr_{\e}$ is the fluid density, $\vu_{\e}$ is the velocity field in $\mathbb R^{3}$, $p_{\e}$ is the pressure, $\mu$ is the viscosity coefficient and is assumed to be a positive constant. The initial data of $(\vr_{\e},\vu_{\e})$ in \eqref{1.3} satisfy 
\begin{eqnarray}\label{1.4}
\left\{
\begin{aligned}
&\vr_{\e}^{0}\in L^{\infty}(\mathbb R^{3}),\quad \vr_{\e}^{0}(x)=\vr_{s}>0 \quad \textup{in } \mathbb R^{3}\backslash\overline{\O}_{\e}, \quad\textup{for some positive constant }\vr_{s},\\
&0<\underline\vr\leq\vr_{\e}^{0}(x)\leq\overline\vr,\quad \vr_{\e}^{0}\rightarrow \vr_{0} \quad \textup{strongly in }L^{1}({\Omega}),\\
&\dive \vu_{\e}^{0} = 0 \quad \textup{in }\O_{\e},\quad \vu_{\e}^{0} = 0\quad \textup{in }\mathbb R^{3}\backslash\overline{\O}_{\e},\quad \vu_{\e}^{0}\rightarrow \vu_{0} \quad \textup{strongly in }L^{2}(\Omega;\mathbb R^{3}).
\end{aligned}
\right .
\end{eqnarray}
The external force $\{\vf_{\e}\}_{\e>0}$ are bounded in $L^{2}((0,T)\times\O;\mathbb R^{3})$. Throughout this paper, we assume $\vu_{\e}$ are extended to be zero outside $\O_{\e}$. Therefore, we can assume that the first equation in \eqref{1.3} (the continuity equation) holds in $\mathbb R^{3}$, as long as $\vr_\e(t,x)=\vr_s$ for $x\in \mathbb R^{3}\backslash\overline{\O}_{\e}$.

\begin{remark}We denote $\dive$,  $\nabla$ and  $\Delta$ as the divergence, the gradient and Laplace operator related to spatial variable $x$ respectively throughout this paper. Moreover, we use $C$ to denote a universal positive constant independent of $\e$.
\end{remark}

\subsection{Finite energy weak solution and renormalized weak solution}

\begin{defi}\label{finite energy weak solution}We call $(\vr_{\e},\vu_{\e})$ a \textbf{weak solution} to system \eqref{1.3}-\eqref{1.4} if:
\begin{itemize}

\item  $(\vr_{\e},\vu_{\e})$ satisfy:
\begin{eqnarray}\label{1.5}
\left\{
\begin{aligned}
&\vr_{\e}\in C([0,T],L^{1}({\Omega_{\varepsilon}})), \\
&\vu_{\e}\in  L^{\infty}(0,T;L^{2}({\Omega_{\varepsilon}};\mathbb R^{3}))\cap L^{2}(0,T;W_{0}^{1,2}({\Omega_{\varepsilon}};\mathbb R^{3})).
\end{aligned}
\right .
\end{eqnarray}

\item for any $\psi\in C_{c}^{1}([0,T)\times {\mathbb R^{3}})$, 
\be\label{1.6}
\int_{0}^{T}\int_{\mathbb R^{3}}\vr_{\e}(\d_{t}\psi+\vu_{\e}\cdot\nabla\psi)dxdt=-\int_{\mathbb R^{3}}\vr_{\e}^{0}\psi(0,x)dx,
\ee
where $\vu_{\e}\equiv 0$ outside $\O_{\e}$.
 \item for any  $\bm{\varphi}\in C_{c}^{1}([0,T)\times \O_{\e};\mathbb R^{3})$, $\dive\bm\varphi =0$, 
\be\label{1.7}
\int_{0}^{T}\int_{\Omega_{\varepsilon}}-\vr_{\e}\vu_{\e}\cdot\d_{t}\bm\varphi-(\vr_{\e} \vu_{\e} \otimes \vu_{\e}):\nabla\bm\varphi+\mu\nabla \vu_{\e}:\nabla \bm\varphi dxdt=\int_{0}^{T}\int_{\Omega_{\varepsilon}}\vr_{\e}\vf_{\e}\cdot \bm\varphi dxdt+\int_{\Omega_{\varepsilon}}\vr_{\e}^{0} \vu_{\e}^{0}\cdot\bm\varphi(0,x)dx.
\ee
\end{itemize}
   
Furthermore, we call $(\vr_{\e},\vu_{\e})$ a \textbf{finite energy weak solution} if the energy inequality
\be\label{1.8}
\frac{1}{2}\int_{\Omega_{\varepsilon}}\vr_{\e} |\vu_{\e}|^{2}(\t,x)dx+\mu\int_{0}^{\t}\int_{\Omega_{\varepsilon}}|\nabla \vu_{\e}|^{2}dxdt\leq \frac{1}{2}\int_{\Omega_{\varepsilon}}\vr_{\e}^{0} |\vu_{\e}^{0}|^{2}(x)dx+\int_{0}^{\t}\int_{\Omega_{\varepsilon}}\vr_{\e}\vf_{\e}\cdot \vu_{\e} dxdt
\ee
holds for \textup{a.a.} $\t\in(0,T)$.

\smallskip

Moreover, a finite energy weak solution $(\vr_{\e},\vu_{\e})$ is called a \textbf{renormalized weak solution} if
\be\label{1.9}
\int_{0}^{T}\int_{\mathbb R^{3}}b(\vr_{\e})\d_{t}\psi+b(\vr_{\e})\vu_{\e}\cdot\nabla\psi dxdt=-\int_{\mathbb R^{3}}b(\vr_{\e}^{0})\psi(0,\cdot)dx
\ee
for any $b\in C^{1}([0,\infty))$ and any $\psi\in C_{c}^{1}([0,T)\times \mathbb R^{3})$. 
\end{defi}

For any fixed $\e>0$, the global existences of the finite energy weak solution and the renormalized weak solution are known in \cite{DiP-L,Lions-C,Lions-D}. 
 \subsection{Main result}

Our approach is based on Stokes' capacity, which is analogous to Newtonian capacity used in the homogenization of elliptic equations. For a compact set $Q\subset B(0,1)$, we introduce 
\be\label{1.10}
C_{j,l}(Q)=\int_{B(0,1)\backslash Q}\nabla \vv^{j}:\nabla \vv^{l}dx,
\ee
where $\vv^{i}$ is the unique solution to the model problem
\begin{eqnarray}\label{1.11}
\left\{
\begin{aligned}
&-\Delta \vv^{i}+\nabla q^{i}=0,\quad \dive \vv^{i}=0 \quad&\mbox{in }&B(0,1)\backslash Q,\\
&\vv^{i}|_{\d Q}=\mathbf e^{i},\quad \vv^{i}|_{\d B(0,1)}=0.
\end{aligned}
\right .
\end{eqnarray}
Here $\{\bm e^{i}\}_{i=1,2,3}$ is the canonical basis of the space $\mathbb R^{3}$. Moreover, we assume
\be
\int_{B(0,1)}q^{i}dx=0.\nn
\ee

We suppose there exists a positive definite symmetric matrix 
\be\label{1.12}
\mathbb C=\{D_{j,l}\}_{j,l=1}^{3},\quad \mathbb C\in L^{\infty}(\O;\mathbb R_{\textup{sym}}^{3\times3}).
\ee
such that at least for a suitable subsequence 
\be
\lim_{\e\rightarrow 0}\sum_{T_{\e,k}\subset B}  C_{j,l}(T_{\e,k})=\int_{B} D_{j,l}(x)dx\nn
\ee
for any Borel set $B\subset \O$. 

 It is shown that the matrix $\mathbb C$ is a constant matrix in the case of periodically distributed holes of identical (rescaled) shape, see Allaire \cite{ALL-NS1}. Now we are ready to give the main result:
\begin{theorem}\label{Theorem}Let $\{\O_{\e}\}_{\e>0}$ be perforated domains specified in \eqref{1.1}-\eqref{1.2}. Suppose
\ba\label{1.13}
\vf_{\e} \rightarrow \vf \quad \mbox{weakly in } L^{2}((0,T)\times\O;\mathbb R^{3}).
\ea
Let $(\vr_{\e},\vu_{\e})$ be  a family of finite energy weak solutions of problem \eqref{1.3} with initial data given in \eqref{1.4}. Then, up to a subsequence, we have 
\ba\label{1.14}
\vu_{\e}\rightarrow \vu \quad\mbox{ weakly-(*) in } L^{\infty}(0,T;L^{2}({\Omega}))\mbox{ and weakly in } L^{2}(0,T;W_{0}^{1,2}({\Omega})),
\ea
\be\label{1.15}
 \vr_{\e}\rightarrow \vr \quad \mbox{in } C([0,T];L^{1}(\O)) \quad \mbox{and} \quad 0<\underline\vr\leq\vr_{\e}\leq\overline\vr.
\ee
Moreover, $(\vr, \vu)$ is a weak solution to the problem
\begin{eqnarray}\label{1.16}
\left\{
\begin{aligned}
&\d_{t}\vr+\dive (\vr \vu) = 0, \quad &\textup{in }&(0,T)\times\Omega,\\
&\d_{t}(\vr \vu)+ \dive (\vr \vu \otimes \vu)-\mu \Delta \vu +\mu\mathbb C \vu+\nabla p =\vr \vf, \quad & \textup{in }&(0,T)\times\Omega,\\
&\dive \vu = 0, \quad &\textup{in }&(0,T)\times\Omega,\\
&\vu = 0, \quad &\textup{on }&(0,T)\times\d\Omega,\\
&\vr|_{t=0} = \vr_{0},\quad (\vr\vu)|_{t=0} = \vr_{0}\vu_{0},
\end{aligned}
\right .
\end{eqnarray}
and satisfies 
\be
0<\underline\vr\leq\vr(x)\leq\overline\vr,\quad \textup{a.a. }(t,x),\nn
\ee
where the initial data $(\vr_{0}, \vu_{0})$ satisfy \eqref{1.4}, $\mu$ is given in \eqref{1.3}.
\end{theorem}

\begin{remark}
The definition of weak solution to system \eqref{1.16}  is similar to Definition \ref{finite energy weak solution}.
\end{remark}

\section{Proof of Theorem \ref{Theorem}}
\subsection{Uniform estimates}
 Recall the classical Poincar\'e inequality: 
\begin{lemma}\label{Poincare inequality}
Let $\O_{\e}$ be the perforated domain defined by \eqref{1.1}-\eqref{1.2}. There exists a constant $C$ independent of $\e$ such that
\be
\|\vu\|_{L^{2}(\O_\e)}\leq C\|\nabla \vu\|_{L^{2}(\O_{\e})},\quad \mbox{for any } \vu\in W_{0}^{1,2}(\O_\e).\nn
\ee
\end{lemma}

By Theorem II.3 of DiPerna-Lions \cite{DiP-L}, the finite energy weak solution $(\vr_\e,\uu_\e)$ of problem \eqref{1.3}-\eqref{1.4} is a renormalized weak solution. We take 
\be
b(\vr_{\e})=[\vr_{\e}-\overline\vr]^{+},\quad b(\vr_{\e})=-[\vr_{\e}-\underline\vr]^{-}\nn
\ee
as test functions in the renormalized equation \eqref{1.9} to deduce
\be\label{2.1}
0<\underline\vr\leq\vr_{\e}(x)\leq\overline\vr,\quad \textup{a.a. }(t,x) \mbox{ uniformly in } \e.
\ee

Recall the energy inequality \eqref{1.8}:
\ba
&\frac{1}{2}\int_{\Omega_{\varepsilon}}\vr_{\e} |\vu_{\e}(\t,x)|^{2}dx+\mu\int_{0}^{\t}\int_{\Omega_{\varepsilon}}|\nabla \vu_{\e}|^{2}dxdt\leq \frac{1}{2}\int_{\Omega_{\varepsilon}}\vr_{\e}^{0} |\vu_{\e}^{0}(x)|^{2}dx+\int_{0}^{\t}\int_{\Omega_{\varepsilon}}\vr_{\e}\vf_{\e}\cdot \vu_{\e} dxdt.\nn
\ea
By \eqref{2.1} and Lemma \ref{Poincare inequality}, we deduce
\ba
&\frac{1}{2}\underline\vr\int_{\Omega_{\varepsilon}}|\vu_{\e}(\t,x)|^{2}dx+\mu\int_{0}^{\t}\int_{\Omega_{\varepsilon}}|\nabla \vu_{\e}|^{2}dxdt\\
&\leq\frac{1}{2}\int_{\Omega_{\varepsilon}}\vr_{\e} |\vu_{\e}|^{2}(\t,x)dx+\mu\int_{0}^{\t}\int_{\Omega_{\varepsilon}}|\nabla \vu_{\e}|^{2}dxdt\\
&\leq\frac{1}{2}\overline\vr\|\vu_{\e}^{0}(x)\|^{2}_{L^{2}({\Omega_{\varepsilon}})}+\overline\vr\|\vu_{\e}\|_{L^{2}(0,\t;W_{0}^{1,2}(\O_\e;\mathbb R^{3}))}\cdot\|\vf_{\e}\|_{L^{2}(0,\t;W^{-1,2}(\O_{\e};\mathbb R^{3}))}\\
&\leq \frac{1}{2}\overline\vr\| \vu_{\e}^{0}(x)\|^{2}_{L^{2}({\Omega_{\varepsilon}})}+C\|\nabla\vu_{\e}\|_{L^{2}(0,\t;L^{2}(\O_\e;\mathbb R^{3\times 3}))}\cdot\|\vf_{\e}\|_{L^{2}(0,\t;W^{-1,2}(\O_{\e};\mathbb R^{3}))}\\
&\leq \frac{\mu}{2}\|\nabla\vu_{\e}\|_{L^{2}(0,\t;L^{2}(\O_\e;\mathbb R^{3\times 3}))}^{2}+C\|\vf_{\e}\|_{L^{2}(0,\t;W^{-1,2}(\O_{\e};\mathbb R^{3}))}^{2}+\frac{\overline\vr}{2}\|\vu_{\e}^{0}(x)\|^{2}_{L^{2}({\Omega_{\varepsilon}})}\nn
\ea
for \textup{a.a.} $\t\in [0,T]$. Since $\{\vf_{\e}\}_{\e>0}$ are bounded in ${L^{2}(0,T;W^{-1,2}(\O;\mathbb R^{3}))}$ and $\{\vu_{\e}^{0}\}_{\e>0}$ are uniformly bounded in $L^{2}(\Omega;\mathbb R^{3})$, we have
\ba
\ess \sup_{t\in [0,T]}\|\vu_{\e}\|_{L^{2}({\Omega_{\varepsilon}})}\leq C,\quad \|\nabla\vu_{\e}\|_{L^{2}(0,T;L^{2}({\Omega_{\varepsilon}}))}\leq C,\nn
\ea
with $\vu_{\e}\equiv 0$ outside $\O_{\e}$, which means (by Poincar\'e inequality)
\ba\label{2.2}
\|\vu_{\e}\|_{L^{\infty}(0,T;L^{2}({\Omega}))}\leq C,\quad \|\nabla\vu_{\e}\|_{L^{2}(0,T;L^{2}({\Omega}))}\leq C,\quad \|\vu_{\e}\|_{L^{2}(0,T;L^{2}({\Omega}))}\leq C.
\ea
Moreover, we have
\be\label{2.3}
\ess \sup_{t\in [0,T]}\int_{\Omega_{\varepsilon}}\vr_{\e} |\vu_{\e}|^{2}(t,x)dx\leq C.
\ee
By \eqref{2.2}, up to a subsequence, we have 
\be\label{2.4}
\vu_{\e}\rightarrow \vu \quad\mbox{ weakly-(*) in } L^{\infty}(0,T;L^{2}({\Omega}))\mbox{ and weakly in } L^{2}(0,T;W_{0}^{1,2}({\Omega})).
\ee

\subsection{Asymptotic limit in continuity equation}
From Lions \cite{Lions-C,Lions-D}, we have the following lemma. 
\begin{lemma}\label{DiPerna-Lions}Let $\vr^{n},\vu^{n}$ satisfy $\vr^{n}\in C([0,T];L^{1}(B_{R}))(\forall ~R>0)$, $\vr^{n}\geq 0$, $\vu^{n}\in L^{2}(0,T;W_0^{1,2}(\O;\mathbb R^{3}))$, where $T\in (0,\infty)$ is fixed. We define $\vr_{0}^{n}=\vr^{n}(0)$ and we assume 
\be
0\leq \vr^{n}\leq C \quad \textup{a.a. in }(0,T)\times\O,\nn
\ee
\be
\dive \vu^{n}=0 \quad \textup{a.a. in } (0,T)\times\O,\quad \|\vu^{n}\|_{L^{2}(0,T;W_{0}^{1,2}(\O))}\leq C,\nn
\ee
\be
\d_{t}\vr^{n}+\dive (\vr^{n} \vu^{n}) = 0 \quad \mbox{in } \mathcal D^{'}( (0,T)\times\mathbb R^{3}),\nn
\ee
\be
\vr_{0}^{n}\rightarrow \vr_{0}\quad \mbox{in } L^{1}(\O),\quad  \vu^{n}\rightarrow \vu \quad \mbox{weakly in } L^{2}(0,T;W_{0}^{1,2}(\O;\mathbb R^{3})),\nn
\ee
for some $\vr_{0}$ satisfying $0\leq \vr_{0}\leq C$. Then we have $\vr^{n}$ converge in $C([0,T];L^{p}(\O))$ for all $1\leq p<\infty$ to the unique solution $\vr$ bounded in $(0,T)\times\O$, of
\begin{eqnarray}\label{2.5}
\left\{
\begin{aligned}
&\d_{t}\vr+\dive (\vr \vu) = 0, \quad \textup{in } \mathcal D^{'}((0,T)\times\mathbb R^{3}),\\
&\vr \in C([0,T];L^{1}(\O)),\quad \vr(0,x) =\vr_{0}(x)  \quad \textup{a.a. in }\O.
\end{aligned}
\right .
\end{eqnarray}
\end{lemma}

By Lemma \ref{DiPerna-Lions},  we have
\be\label{2.6}
\int_{0}^{T}\int_{\mathbb R^{3}}(\vr(t,x)\d_{t}\varphi(t,x)+\vr(t,x)\vu(t,x)\cdot\nabla\varphi(t,x))dxdt=0, \quad \forall~\varphi\in C_{c}^{\infty}((0,T)\times\mathbb R^{3}).
\ee
Taking 
\be
\varphi(t,x)=\psi(t)\phi(x), \quad \psi\in C_c^\infty(0,T), \quad \phi\in C_c^\infty(\mathbb R^{3})\nn
\ee
as a test function, we may infer that the function
\be
t\mapsto\int_{\mathbb R^{3}}\vr(t,x) \phi(x) dx\nn
\ee
is absolutely continuous on $[0,T]$ for any $\phi\in C_c^\infty(\mathbb R^{3})$. 
The initial datum in \eqref{2.5} is satisfied in the sense that
\ba
\lim_{t\rightarrow 0^{+}}\int_{\mathbb R^{3}}\vr(t,x)\phi(x) dx
=\int_{\mathbb R^{3}}\vr_0(x)\phi(x) dx,\quad\text{for any } \phi\in C_c^\infty(\mathbb R^{3}).\nn
\ea

Let $\psi_{\e}\in C_{c}^{\infty}(0,T)$ and 
\be
0\leq \psi_{\e}\leq 1, \quad \psi_{\e}\nearrow 1_{[0,T]}, \quad \mbox{as } \e \rightarrow 0.\nn
\ee
For any $\varphi\in C_{c}^{\infty}([0,T)\times\mathbb R^{3})$, we take 
\be
\varphi_{\e}(t,x)=\psi_{\e}(t)\varphi(t,x),\quad \varphi\in C_{c}^{\infty}([0,T)\times\mathbb R^{3}),\nn
\ee
as a test function in \eqref{2.6}. Then, we have
\be\label{2.7}
\int_{0}^{T}\int_{\mathbb R^{3}}(\vr(t,x)\d_{t}(\psi_{\e}(t)\varphi(t,x))+\vr(t,x)\vu(t,x)\cdot\nabla(\psi_{\e}(t)\varphi(t,x)))dxdt=0.
\ee
Letting $\e \rightarrow 0$ in \eqref{2.7}, we conclude that
\be\label{2.8}
-\int_{\mathbb R^{3}}\vr_{0}\varphi(0,x)dx=\int_{0}^{T}\int_{\mathbb R^{3}}\vr(\d_{t}\varphi+\vu\cdot\nabla\varphi)dxdt
\ee
 for any $\varphi\in C_{c}^{1}([0,T)\times \mathbb R^{3})$. By \eqref{2.8}, we obtain that the limit $(\vr,\vu)$ satisfies the weak formulation of continuity equation. 

DiPerna-Lions theory again shows that $(\vr,\vu)$ satisfy the equation in the renormalized sense \eqref{1.9}, which implies
\be\label{2.9}
0<\underline\vr\leq\vr\leq\overline\vr,\quad \textup{a.a. }(t,x).
\ee
By Lemma \ref{DiPerna-Lions}, we have
\be\label{2.10}
\vr_{\e}\rightarrow \vr \quad\mbox{in } C([0,T];L^{q}(\O))\mbox{ for any } 1\leq q<\infty.
\ee
Thanks to \eqref{2.1}, \eqref{2.9} and \eqref{2.10}, we have
\be\label{2.11}
\sqrt{\vr_{\e}}\rightarrow \sqrt{\vr} \quad\mbox{weakly-(*) in }L^{\infty}([0,T]\times \O)\mbox{ and strongly in }C([0,T];L^{q}(\O))\mbox{ for any }1\leq q<\infty.
\ee

\subsection{Compactness in time of $\sqrt{\vr_\e}\vu_\e$}

 For $\mu>0$, we consider stationary Stokes problem in the form:
\begin{eqnarray}\label{2.12}
\left\{
\begin{aligned}
&-\mu\Delta \vv_{\e}+\nabla q_{\e}= \vf_{\e},  \quad&\mbox{in }&\O_{\e},\\
&\dive \vv_{\e}=0, \quad&\mbox{in }&\O_{\e},\\
&\vv_\e|_{\d\O_{\e}}=0.
\end{aligned}
\right .
\end{eqnarray}
Denoting $L_{0}^{2}(\O_{\e}):=\{f\in L^{2}(\O_{\e}): \int_{\O_{\e}}f dx=0\}$, we say that $(\vv_{\e},q_{\e})\in W_{0}^{1,2}(\O_{\e};\mathbb R^{3})\times L_{0}^{2}(\O_{\e})$ is a weak solution to problem \eqref{2.12} if the following integral identities hold:
\be
\int_{\O_{\e}} \nabla \vv_{\e}: \nabla\vh_{\e}dx-\int_{\O_{\e}} q_{\e} \nabla\cdot\vh_{\e}dx=\int_{\O_{\e}}\vf_{\e}\cdot\vh_{\e}dx,\quad \forall~\vh_{\e}\in W_{0}^{1,2}(\O_{\e};\mathbb R^{3}),\nn
\ee
\be
\int_{\O_{\e}} r_{\e} \nabla\cdot\vv_{\e}dx=0,\quad \forall~ r_{\e}\in L_{0}^{2}(\O_{\e}).\nn
\ee
It is known that the problem \eqref{2.12} admits a unique weak solution $(\vv_{\e},q_{\e})$ for any $\vf_{\e}\in W^{-1,2}(\O;\mathbb R^{3})$, $\e>0$ fixed. All functions in \eqref{2.12} defined in $\O_{\e}$ are extended to be zero in $\O\backslash \O_{\e}$. Particularly, $\vf_{\e}$ can be viewed as a functional in $W^{-1,2}(\O_{\e};\mathbb R^{3})$. 

Similarly as Proposition 5.1 in \cite{FeNaNe},  we have the following limit result for the stationary Stokes problem in \eqref{2.12}.

\begin{proposition}\label{Stationary stokes problem} Assume that
\be\label{2.13}
\vf_{\e}\rightarrow \ff \quad \mbox{in } W^{-1,2}(\O;\mathbb R^{3}).
\ee
Let $(\vv_{\e},q_{\e})$ be the unique weak solution to the stationary Stokes problem \eqref{2.12} in $\O_{\e}$. Then, at least for a suitable subsequence,
\be\label{2.14}
\vv_{\e}\rightarrow \vv \quad \mbox{weakly in } W_{0}^{1,2}(\O;\mathbb R^{3}), \quad q_{\e}\rightarrow q \quad \mbox{weakly in } L^{2}(\O),
\ee
where $(\vv,q)$ is the unique weak solution to the problem
\begin{eqnarray}\label{2.15}
\left\{
\begin{aligned}
&-\mu\Delta \vv+\mu\mathbb C\vv+\nabla q= \vf,  \quad&\mbox{in }&\O,\\
&\dive \vv=0, \quad&\mbox{in }&\O,\\
&\vv|_{\d\O}=0,
\end{aligned}
\right .
\end{eqnarray}
with the matrix $\mathbb C$ determined in \eqref{1.12}.
\end{proposition}
\begin{lemma}\label{Compactness in time} Let $(\vr_{\e},\vu_{\e})$ be a family of weak solutions of problem \eqref{1.3} with initial data given in \eqref{1.4}. Then, we have 
\be
\sqrt{\vr_\e}\vu_\e\to\sqrt{\vr}\vu \quad\mbox{strongly in } L^{2}((0,T)\times\O;\mathbb R^{3})\mbox{ as } \e \rightarrow 0.\nn
\ee
\end{lemma}
\begin{proof}
Given $\ww\in C_{c}^{\infty}(\O;\mathbb R^{3})$, $\dive \ww=0$, let
\be
\vf=-\mu\Delta \ww+\mu\mathbb C\ww \in L^{\infty}(\O;\mathbb R^{3}).\nn
\ee
Let $\vf_{\e}=1_{\O_{\e}}\vf$, and $\ww_{\e}$ be the unique solution to Stokes problem
\be
-\mu\Delta \ww_{\e}+\nabla P_{\e}=\vf_{\e}, \quad \dive \ww_{\e}=0\quad \mbox{in }  \O_{\e},  \quad \ww_{\e}\in W_{0}^{1,2}(\O_{\e};\mathbb R^{3}).\nn
\ee
 By Proposition \ref{Stationary stokes problem}, up to a subsequence, we have
\be
\ww_{\e}\rightarrow \ww \quad \mbox{weakly in } W_{0}^{1,2}(\O;\mathbb R^{3}).\nn
\ee
Moreover, by Rellich-Kondrachov theorem, the Sobolev space $W^{1,2}(\O)$ is compactly embedded in $L^{2}(\O)$. Thus we have
\be
\ww_{\e}\rightarrow \ww \quad \mbox{strongly in } L^{2}(\O;\mathbb R^{3}).\nn
\ee
Combined with \eqref{2.1} and \eqref{2.2}, this gives
\be\label{2.16}
\ess\sup_{t\in[0,T]}\int_{\O}\vr_{\e} \vu_{\e}\cdot (\ww_{\e}-\ww) dx\rightarrow 0 \quad \mbox{as } \e \rightarrow 0.
\ee

Taking 
\be\label{2.17}
\bm\varphi(t,x)=\psi(t)\ww_{\e}(x), \quad \psi(t)\in C_{c}^{\infty}(0,T)
\ee
 in \eqref{1.7}, by Arzel\`a-Ascoli theorem, we deduce that the family of functions 
\be\label{2.18}
[t\mapsto\int_{\O}\vr_{\e} \vu_{\e}\cdot \ww_{\e} dx] \quad \mbox{is precompact in } C([0,T]).
\ee

Recalling \eqref{2.4} and \eqref{2.10}, we have
\be
\vu_{\e}\rightarrow \vu \quad\mbox{ weakly-(*) in } L^{\infty}(0,T;L^{2}({\Omega})),\nn
\ee
and 
\be
\vr_{\e}\rightarrow \vr \quad\mbox{in } C([0,T];L^{2}(\O)).\nn
\ee
Since
\be\label{2.19}
[t\mapsto\int_{\O}\vr_{\e} \vu_{\e}\cdot \ww dx]=[t\mapsto\int_{\O}\vr_{\e} \vu_{\e}\cdot (\ww-\ww_{\e}) dx]+[t\mapsto\int_{\O}\vr_{\e} \vu_{\e}\cdot \ww_{\e} dx],
\ee
combining \eqref{2.16} and \eqref{2.18}, we conclude that
\be\label{2.20}
[t\mapsto\int_{\O}\vr_{\e} \vu_{\e}\cdot \ww dx]\rightarrow [t\mapsto\int_{\O}\vr \vu\cdot \ww dx] \quad \mbox{in } L^{\infty}(0,T) \mbox{ for any }\ww\in C_{c}^{\infty}(\O;\mathbb R^{3}),\quad \dive \ww=0.
\ee
Using the density of smooth compactly supported functions in $W_{0,\dive}^{1,2}(\O;\mathbb R^{3})$-the Sobolev space $W_{0}^{1,2}(\O;\mathbb R^{3})$ of solenoidal vector fields, we deduce
\be\label{2.21}
\vr_{\e}\vu_{\e}\rightarrow \vr\vu \quad \mbox{in } L^{q}([0,T];W_{\dive}^{-1,2}(\O;\mathbb R^{3})), \quad 1\leq q<\infty.
\ee
Thanks to \eqref{2.4} and \eqref{2.21}, we have
\be\label{2.22}
\int_{0}^{T}\int_{\O}\vr_{\e} |\vu_{\e}|^{2}dxdt\rightarrow \int_{0}^{T}\int_{\O}\vr |\vu|^{2}dxdt \quad \mbox{as } \e \rightarrow 0.
\ee

By the property that $\vu_{\e}=0$ in $\O\backslash \O_{\e}$, using \eqref{2.3}, we have
\be\label{2.23}
\int_{0}^{T}\int_{\Omega}\vr_{\e} |\vu_{\e}|^{2}dxdt\leq C.
\ee
Thus, up to a subsequence, we have
\be
\int_{0}^{T}\int_{\O}\sqrt{\vr_\e}\vu_\e(t,x)\cdot \bm{\tilde\varphi} dxdt\rightarrow\int_{0}^{T}\int_{\O}\overline{\sqrt{\vr}\vu}\cdot \bm{\tilde\varphi} dxdt \quad \mbox{as } \e \rightarrow 0,\nn
\ee
for any $\bm{\tilde\varphi}\in{L^{2}((0,T)\times\O;\mathbb R^{3})}$, where $\overline{\sqrt{\vr}\vu}$ denotes a weak  limit of $\sqrt{\vr_\e}\vu_\e$ in $L^{2}((0,T)\times\O;\mathbb R^{3})$. 

By \eqref{2.4} and \eqref{2.11}, we have
\be
\int_{0}^{T}\int_{\O}\sqrt{\vr_\e}\vu_\e(t,x)\cdot \bm\varphi dxdt\rightarrow\int_{0}^{T}\int_{\O}\sqrt{\vr}\vu\cdot \bm\varphi dxdt \quad \mbox{as } \e \rightarrow 0,\nn
\ee
for all $\bm\varphi\in C_{c}^{\infty}((0,T)\times \O;\mathbb R^{3})$. Thus
\be\label{2.24}
\sqrt{\vr_\e}\vu_\e\to\sqrt{\vr}\vu \quad\mbox{weakly in } L^{2}((0,T)\times\O;\mathbb R^{3})\mbox{ as } \e \rightarrow 0.
\ee
Then, \eqref{2.22} and \eqref{2.24} yield 
\be\label{2.25}
\sqrt{\vr_\e}\vu_\e\to\sqrt{\vr}\vu \quad\mbox{strongly in } L^{2}((0,T)\times\O;\mathbb R^{3})\mbox{ as } \e \rightarrow 0.
\ee
\end{proof}
 \subsection{Asymptotic limit of the momentum equation}
 Now we consider the asymptotic limit of the momentum equation (second equation in \eqref{1.3}). We use the time regularization by means of a convolution with a family of regularization kernels $\o_{\delta}=\o_{\delta}(t)$ satisfying:
\be
\o_{\delta}(t)=\frac{1}{\delta}\o(\frac{t}{\delta}),\quad \o\in C_{c}^{\infty}(-1,1),\quad \o\geq 0,\quad \o(-z)=\o(z),\quad \o'(z)\leq 0 \mbox{ for } z\geq 0,\nn
\ee
and
\be
\int_{-1}^{1}\o(z)dz=1.\nn
\ee

 For any $\bm{\tilde\phi}\in W_{0}^{1,2}(\O_{\e};\mathbb R^{3})$ with $\dive \bm{\tilde\phi}=0$, consider $\o_{\delta}(\t-t)\bm{\tilde\phi}(x)$ as a test function in the weak formulation \eqref{1.7}. Denoting $[v]_{\delta}=\o_{\delta}*v$, we get 
\be\label{2.26}
\mu\int_{\O_{\e}}\nabla[\vu_{\e}(\t,\cdot)]_{\delta}:\nabla \bm{\tilde\phi} dx=\int_{\O_{\e}}[\vr_{\e}\vf_{\e}(\t,\cdot)]_{\delta}\cdot \bm{\tilde\phi} dx+\int_{\O_{\e}}[\vr_{\e}\vu_{\e}\otimes \vu_{\e}(\t,\cdot)]_{\delta}:\nabla \bm{\tilde\phi} dx-\int_{\O_{\e}}\d_{\t}[\vr_{\e}\vu_{\e}(\t,\cdot)]_{\delta}\cdot\bm{\tilde\phi} dx,
\ee
for any $\t\in (\delta, T-\delta)$. 

Now, we verify the convergence of $\d_{\t}[\vr_{\e}\vu_{\e}(\t,\cdot)]_{\delta}$, $[\vr_{\e}\vf_{\e}(\t,\cdot)]_{\delta}$ and $[\vr_{\e}\vu_{\e}\otimes \vu_{\e}(\t,\cdot)]_{\delta}$ for any fixed $\t\in(\delta,T-\delta)$.
\begin{proposition}\label{convergence} Let $(\vr_{\e},\vu_{\e})$ be a family of weak solutions of problem \eqref{1.3} with initial data given in \eqref{1.4}. For any $\t\in(\delta,T-\delta)$, we have
\ba
&\d_{\t}[\vr_{\e}\vu_{\e}(\t,\cdot)]_{\delta}\rightarrow\d_{\t}[\vr\vu(\t,\cdot)]_{\delta} \quad\mbox{in } W^{-1,2}(\O;\mathbb R^{3}) \mbox{ as } \e \rightarrow 0,\\
&[\vr_{\e}\vf_{\e}(\t,\cdot)]_{\delta}\rightarrow[\vr \vf(\t,\cdot)]_{\delta} \quad\mbox{in } W^{-1,2}(\O;\mathbb R^{3}) \mbox{ as } \e \rightarrow 0,\\
&[\dive(\vr_{\e}\vu_{\e}\otimes \vu_{\e})(\t,\cdot)]_{\delta}\rightarrow[\dive(\vr\vu\otimes \vu)(\t,\cdot)]_{\delta} \quad\mbox{in } W^{-1,2}(\O;\mathbb R^{3}) \mbox{ as } \e \rightarrow 0.\nn
\ea
\end{proposition}

\begin{proof}
For $\d_{\t}[\vr_{\e}\vu_{\e}(\t,\cdot)]_{\delta}$, by the weak convergence of $\sqrt{\vr_{\e}}$  in \eqref{2.11} and strong convergence of $\sqrt{\vr_{\e}}\vu_{\e}$ in Lemma \ref{Compactness in time}, we obtain, for any $\bm\phi\in W_{0}^{1,2}(\O;\mathbb R^{3})$,
\ba\label{2.27}
&\int_{\O_{\e}}\d_{\t}[\vr_{\e}\vu_{\e}(\t,x)]_{\delta}\cdot\bm\phi dx=\int_{0}^{T}\int_{\O}\vr_{\e}\vu_{\e}(t,x)\cdot\d_{\t}\o_{\delta}(\t-t)\bm\phi dxdt\\
&\rightarrow\int_{0}^{T}\int_{\O}\vr\vu(t,x)\cdot\d_{\t}\o_{\delta}(\t-t)\bm\phi dxdt=\int_{\O}\d_{\t}[\vr\vu(\t,x)]_{\delta}\cdot\bm\phi dx \quad \mbox{as } \e \rightarrow 0,
\ea
for any $\t\in(\delta,T-\delta)$. Thus we have
\be\label{2.28}
\d_{\t}[\vr_{\e}\vu_{\e}(\t,\cdot)]_{\delta}\rightarrow\d_{\t}[\vr\vu(\t,\cdot)]_{\delta} \quad\mbox{in } W^{-1,2}(\O;\mathbb R^{3}) \mbox{ as } \e \rightarrow 0,
\ee
for any $\t\in(\delta,T-\delta)$.

For $[\vr_{\e}\vf_{\e}(\t,\cdot)]_{\delta}$, since $\{\vr_{\e}\}_{\e>0}$ are uniformly bounded in $L^{\infty}((0,T)\times\O)$ and $\{\vf_{\e}\}_{\e>0}$ are uniformly bounded in $ L^{2}((0,T)\times\O;\mathbb R^{3})$, we have 
\be
\|\vr_{\e}\vf_{\e}\|_{L^{2}((0,T)\times\O)}\leq C.\nn
\ee
Thus, up to a subsequence, we have
\be\label{2.29}
\int_{0}^{T}\int_{\O}\vr_{\e}\vf_{\e}(t,x)\cdot \bm{\tilde\varphi} dxdt\rightarrow\int_{0}^{T}\int_{\O}\overline{\vr\vf}\cdot \bm{\tilde\varphi} dxdt \quad \mbox{as } \e \rightarrow 0,
\ee
for any $\bm{\tilde\varphi}\in{L^{2}((0,T)\times\O;\mathbb R^{3})}$, where $\overline{\vr\vf}$ denotes a weak  limit of ${\vr_{\e}\vf_{\e}}$ in $L^{2}((0,T)\times\O)$. 

Using \eqref{2.10}  and  \eqref{1.13}, we have

\be\label{2.30}
\int_{0}^{T}\int_{\O}\vr_{\e}\vf_{\e}(t,x)\cdot \bm\varphi dxdt\rightarrow\int_{0}^{T}\int_{\O}\vr\vf\cdot \bm\varphi dxdt \quad \mbox{as } \e \rightarrow 0,
\ee
for all $\bm\varphi\in C_{c}^{\infty}((0,T)\times \O;\mathbb R^{3})$. For any $\bm\phi\in W_{0}^{1,2}(\O;\mathbb R^{3})$, we have 
\be
\o_{\delta}(\t-t) \bm\phi\in {L^{2}(0,T;W_{0}^{1,2}(\O;\mathbb R^{3}))}.\nn
\ee
It follows from \eqref{2.29} and \eqref{2.30} that
\be
\int_{0}^{T}\int_{\O}\vr_{\e}\vf_{\e}(t,x)\cdot\o_{\delta}(\t-t)\bm\phi dxdt\rightarrow\int_{0}^{T}\int_{\O}\vr\vf(t,x)\cdot\o_{\delta}(\t-t)\bm\phi dxdt \quad \mbox{as } \e \rightarrow 0,\nn
\ee
which means 
\be\label{2.31}
[\vr_{\e}\vf_{\e}(\t,\cdot)]_{\delta}\rightarrow[\vr \vf(\t,\cdot)]_{\delta} \quad\mbox{in } W^{-1,2}(\O;\mathbb R^{3}) \mbox{ as } \e \rightarrow 0,
\ee
for any $\t\in(\delta,T-\delta)$.

As for $[\vr_{\e}\vu_{\e}\otimes \vu_{\e}(\t,\cdot)]_{\delta}$, by  \eqref{2.4} and \eqref{2.10}, we know that $\{\vr_{\e}\}_{\e>0}$ are uniformly bounded in $L^{\infty}((0,T)\times\O)$ and $\{\vu_{\e}\}_{\e>0}$ are uniformly bounded in $ {L^{2}(0,T;L^{4}(\O;\mathbb R^{3}))}$. Thus 
\be
\|\vr_{\e}\vu_{\e}\otimes \vu_{\e}\|_{L^{1}(0,T;L^{2}(\O))}\leq C.\nn
\ee
 Up to a subsequence, we have
 \be\label{2.32}
\int_{0}^{T}\int_{\O}\vr_{\e}\vu_{\e}\otimes \vu_{\e}(t,x): \nabla\bm{\tilde\varphi} dxdt\rightarrow\int_{0}^{T}\int_{\O}\overline{\vr\vu\otimes \vu}:\nabla \bm{\tilde\varphi} dxdt \quad \mbox{as } \e \rightarrow 0,
\ee
for any $ \bm{\tilde\varphi}\in{L^{\infty}(0,T;W^{1,2}(\O;\mathbb R^{3}))}$, where $\overline{\vr\vu\otimes \vu}$ denotes a weak  limit of $\vr_{\e}\vu_{\e}\otimes \vu_{\e}$ in $L^{1}(0,T;L^{2}(\O;\mathbb R^{3\times 3}))$. By Lemma \ref{Compactness in time}, we have
\ba\label{2.33}
&\int_{0}^{T}\int_{\O}\vr_{\e}\vu_{\e}\otimes \vu_{\e}(t,x):\nabla \bm\varphi dxdt=\int_{0}^{T}\int_{\O}\sqrt{\vr_\e}\vu_\e\otimes \sqrt{\vr_\e}\vu_\e(t,x):\nabla \bm\varphi dxdt\\
&\rightarrow\int_{0}^{T}\int_{\O}\sqrt{\vr}\vu\otimes \sqrt{\vr}\vu(t,x):\nabla \bm\varphi dxdt=\int_{0}^{T}\int_{\O}\vr\vu\otimes \vu:\nabla \bm\varphi dxdt \quad \mbox{as } \e \rightarrow 0,
\ea
for all $\bm\varphi\in C_{c}^{\infty}((0,T)\times \O;\mathbb R^{3})$. For any $\bm\phi\in W^{1,2}(\O;\mathbb R^{3})$, we have
\be
\o_{\delta}(\t-t)\nabla \bm\phi\in {L^{\infty}(0,T;L^{2}(\O;\mathbb R^{3}))},\quad\mbox{for any }\t\in(\delta,T-\delta).\nn
\ee
Then 
\be\label{2.34}
\int_{0}^{T}\int_{\O}\vr_{\e}\vu_{\e}\otimes \vu_{\e}(t,x):\o_{\delta}(\t-t)\nabla \bm\phi dxdt\rightarrow\int_{0}^{T}\int_{\O}{\vr\vu\otimes \vu}(t,x):\o_{\delta}(\t-t)\nabla \bm\phi  dxdt \quad \mbox{as } \e \rightarrow 0,
\ee
for any $\bm\phi\in W^{1,2}(\O;\mathbb R^{3})$, which yields
\be\label{2.35}
[\dive(\vr_{\e}\vu_{\e}\otimes \vu_{\e})(\t,\cdot)]_{\delta}\rightarrow[\dive(\vr\vu\otimes \vu)(\t,\cdot)]_{\delta} \quad\mbox{in } W^{-1,2}(\O;\mathbb R^{3}) \mbox{ as } \e \rightarrow 0,
\ee
for any $\t\in(\delta,T-\delta)$. 
\end{proof}
\smallskip 

Now we are in the position to prove Theorem \ref{Theorem}.
\begin{proof}[Proof of Theorem \ref{Theorem}]
For any $\bm{\tilde\phi}\in W_{0}^{1,2}(\O_{\e};\mathbb R^{3})$ with $\dive \bm{\tilde\phi}=0$,
combining \eqref{2.26}, Propositions \ref{Stationary stokes problem} and \ref{convergence}, we have
\ba\label{2.36}
&\mu\int_{\O}\nabla[\vu(\t,\cdot)]_{\delta}:\nabla \bm{\tilde\phi} dx+\mu\int_{\O}(\mathbb C[\vu(\t,\cdot)]_{\delta})\cdot \bm{\tilde\phi} dx\\
&=\int_{\O}[\vr \vf(\t,\cdot)]_{\delta}\cdot\bm{\tilde\phi} dx+\int_{\O}[\vr\vu\otimes \vu(\t,\cdot)]_{\delta}:\nabla \bm{\tilde\phi} dx-\int_{\O}\d_{t}[\vr\vu(\t,\cdot)]_{\delta}\cdot\bm{\tilde\phi} dx
\ea
for any $\t\in(\delta,T-\delta)$. For any fixed $\delta_{1}>0$ small, any $\delta\in(0,\delta_{1})$, and any $\psi\in C_{c}^{\infty}([0,T))$, we have
\ba\label{2.37}
&\mu\int_{\delta_{1}}^{T-\delta_{1}}\int_{\O}\nabla[\vu(t,x)]_{\delta}:\nabla \bm{\tilde\phi}(x)\psi(t) dxdt+\mu\int_{\delta_{1}}^{T-\delta_{1}}\int_{\O}(\mathbb C[\vu(t,x)]_{\delta})\cdot \bm{\tilde\phi}(x)\psi(t) dxdt \\
&=\int_{\delta_{1}}^{T-\delta_{1}}\int_{\O}[\vr \vf(t,x)]_{\delta}\cdot\bm{\tilde\phi}(x)\psi(t) dxdt+\int_{\delta_{1}}^{T-\delta_{1}}\int_{\O}[\vr\vu\otimes \vu(t,x)]_{\delta}:\nabla \bm{\tilde\phi}(x) \psi(t)dxdt\\
&-\int_{\delta_{1}}^{T-\delta_{1}}\int_{\O}\d_{t}[\vr\vu(t,x)]_{\delta}\cdot\bm{\tilde\phi}(x)\psi(t) dxdt.
\ea

 Since functions of the form
\be
\sum_{|k|\leq N} \bm\phi_k(x)\psi_k(t), \quad \text{with } \bm\phi_k\in C_{c}^{\infty}(\Omega;\mathbb{R}^3), \quad\dive \bm\phi_k=0,\quad \psi_k\in C_{c}^{\infty}([0,T)),\nn
\ee
are dense in 
\be\label{2.38}
A=\{\bm\varphi\in C_{c}^{1}([0,T)\times \O;\mathbb R^{3})|\ \dive \bm\varphi=0\},
\ee
 \eqref{2.37} holds with $\bm{\tilde\phi}(x)\psi(t)$ replaced by $\bm\varphi(t,x)$ and $\nabla \bm{\tilde\phi}(x)\psi(t)$ replaced by $\nabla\bm\varphi(t,x)$  for any $\bm\varphi\in A$.
Letting $\delta\rightarrow 0$, we have
\ba\label{2.39}
&\int_{\delta_{1}}^{T-\delta_{1}}\int_{\O}(-\vr\vu\cdot\d_{t}\bm\varphi-\vr\vu\otimes \vu:\nabla \bm\varphi+\mu\nabla\vu:\nabla \bm\varphi+\mu\mathbb C\vu\cdot \bm\varphi) dxdt\\
&=\int_{\delta_{1}}^{T-\delta_{1}}\int_{\O}\vr \vf\cdot \bm\varphi dxdt+\int_{\O}\vr(\delta_1,x)\vu(\delta_1,x)\cdot\bm\varphi(\delta_1,x) dx-\int_{\O}\vr(T-\delta_1,x)\vu(T-\delta_1,x)\cdot\bm\varphi(T-\delta_1,x) dx
\ea
for any test function 
\be
\bm\varphi\in C_{c}^{1}([0,T)\times \O;\mathbb R^{3}),\quad \dive \bm\varphi=0.\nn
\ee
By \eqref{2.18}, up to a subsequence, $F_{\e}(t):=[t\mapsto\int_{\O}\vr_{\e} \vu_{\e}\cdot \ww_{\e} dx]$ uniformly converge to  $F(t):=[t\mapsto\int_{\O}\vr\vu\cdot \ww dx]$, so we have
\ba
&\lim_{t\rightarrow 0^{+}}\int_{\O}\vr(t,x)\vu(t,x)\cdot\bm\ww(x) dx=\lim_{t\rightarrow 0^{+}}\lim_{\e\rightarrow 0}\int_{\O}\vr_{\e}(t,x)\vu_{\e}(t,x)\cdot\bm\ww_{\e}(x) dx\\
&=\lim_{\e\rightarrow 0}\lim_{t\rightarrow 0^{+}}\int_{\O}\vr_{\e}(t,x)\vu_{\e}(t,x)\cdot\bm\ww_{\e}(x) dx=\lim_{\e\rightarrow 0}\int_{\O}\vr_{\e}^{0}(x)\vu_{\e}^{0}(x)\cdot\bm\ww_{\e}(x) dx\\
&=\int_{\O}\vr_0(x)\vu_0(x)\cdot\ww(x) dx,\quad\text{for any } \ww\in C_c^\infty(\O),\quad\dive \ww=0.\nn
\ea
Thus
\ba
&\lim_{t\rightarrow 0}\int_{\O}\vr(t,x)\vu(t,x)\cdot\bm\phi(x)\psi(t) dx\\
&=\lim_{t\rightarrow 0}\int_{\O}\vr(t,x)\vu(t,x)\cdot\bm\phi(x)\psi(0) dx+\lim_{t\rightarrow 0}\int_{\O}\vr(t,x)\vu(t,x)\cdot\bm\phi(x)(\psi(t)-\psi(0)) dx\\
&=\int_{\O}\vr_0(x)\vu_0(x)\cdot\bm\phi(x)\psi(0) dx,\quad\text{for any } \bm\phi\in C_c^\infty(\O),~~\dive \bm\phi=0, \quad \psi\in C_c^\infty([0,T)).\nn
\ea
Again by the property that the functions of the form
\be
\sum_{|k|\leq N}\bm\phi_k(x)\psi_k(t), \quad \text{with } \bm\phi_k\in C_{c}^{\infty}(\Omega;\mathbb{R}^3), \quad\dive \bm\phi_{k}=0,\quad \psi_k\in C_{c}^{\infty}([0,T)) ,\nn
\ee
are dense in $A$, we have
\ba\label{2.40}
&\lim_{t\rightarrow 0^{+}}\int_{\O}\vr(t,x)\vu(t,x)\cdot\bm\varphi(t,x) dx=\int_{\O}\vr_0(x)\vu_0(x)\cdot\bm\varphi(0,x) dx,
\ea
for any $\bm\varphi\in C_{c}^{1}([0,T)\times \O;\mathbb R^{3})$, $\dive \bm\varphi=0$. Similarly, we have
\ba\label{2.41}
&\lim_{t\rightarrow T^{-}}\int_{\O}\vr(t,x)\vu(t,x)\cdot\bm\varphi(t,x) dx=\int_{\O}\vr(T,x)\vu(T,x)\cdot\bm\varphi(T,x) dx=0,
\ea
for any $\bm\varphi\in C_{c}^{1}([0,T)\times \O;\mathbb R^{3})$, $\dive \bm\varphi=0$. 

Letting $\delta_{1}\rightarrow 0$ in \eqref{2.39}, combining \eqref{2.40} and \eqref{2.41}, we obtain
\be\label{2.42}
\int_{0}^{T}\int_{\O}(-\vr\vu\cdot\d_{t}\bm\varphi-\vr\vu\otimes \vu:\nabla \bm\varphi+\mu\nabla\vu:\nabla \bm\varphi+\mu\mathbb C\vu\cdot \bm\varphi) dxdt=\int_{0}^{T}\int_{\O}\vr \vf\cdot \bm\varphi dxdt+\int_{\O}\vr_{0}\vu_{0}\cdot\bm\varphi(0,\cdot) dx
\ee
for any function 
\be
\bm\varphi\in C_{c}^{1}([0,T)\times \O;\mathbb R^{3}),\quad \dive \bm\varphi=0.\nn
\ee
\eqref{2.8} and \eqref{2.42} show that the limit $(\vr,\vu)$ is a weak solution to problem \eqref{1.16}. 

Combining \eqref{2.4}, \eqref{2.8}-\eqref{2.10} and \eqref{2.42}, we complete the proof of Theorem \ref{Theorem}.
\end{proof}

\noindent
{\bf Acknowledgements.} The author is partially supported by the NSF of China under Grant 12171235. The author is grateful to professor Yong Lu and professor Milan Pokorn\'y for fruitful discussions. \\
{\bf Conflict of interest} The author declares no conflict of interest in this paper.

\end{document}